\documentclass[a4paper,11pt,reqno]{amsart}
\usepackage{amsfonts}
\usepackage{amsmath}
\usepackage{logicproof}
\usepackage{amssymb, amsthm}
\usepackage{tabularx}
\usepackage{caption}
\usepackage{booktabs}
\usepackage{multirow,dsfont}
\usepackage{array}
\usepackage{centernot}
\usepackage{enumitem}
\newcolumntype{L}[1]{>{\raggedright\let\newline\\\arraybackslash\hspace{0pt}}m{#1}}
\newcolumntype{C}[1]{>{\centering\let\newline\\\arraybackslash\hspace{0pt}}m{#1}}
\newcolumntype{R}[1]{>{\raggedleft\let\newline\\\arraybackslash\hspace{0pt}}m{#1}}
\usepackage[symbol]{footmisc}
\usepackage[labelfont=bf,format=plain,justification=raggedright,singlelinecheck=false]{caption}
\usepackage{mathrsfs}
\usepackage[]{mdframed}
\usepackage[colorlinks]{hyperref}
\usepackage{tcolorbox}
\usepackage{adjustbox}
\tcbuselibrary{theorems}
\newtcbtheorem[number within=section]{mytheo}{Note}%
{colback=white!5,colframe=black!35!black,fonttitle=\bfseries}{th}
\numberwithin{equation}{section}

\usepackage{graphicx}
\usepackage{textgreek}

 {
      \theoremstyle{plain}
      \newtheorem{assumption}{Assumption}
  }
\newtheorem{theorem}{Theorem}[section]
\newtheorem{definition}[theorem]{Definition}
\newtheorem{remark}[theorem]{Remark}
\newtheorem{corollary}[theorem]{Corollary}

\newtheorem{lemma}[theorem]{Lemma}
\newtheorem{example}[theorem]{Example}

\begin{document}

\title{Revisiting general source condition in learning over a Hilbert space}

\author[N. Gupta]{Naveen Gupta}
\address[N. Gupta]{Indian Institute of Technology Delhi, India}
\email{ngupta.maths@gmail.com}
\author{S. Sivananthan}
\address[S. Sivananthan]{Indian Institute of Technology Delhi, India}
\email{siva@maths.iitd.ac.in}
\maketitle

\begin{abstract}
In Learning Theory, the smoothness assumption on the target function (known as source condition) is a key factor in establishing theoretical convergence rates for an estimator. The existing general form of the source condition, as discussed in learning theory literature, has traditionally been restricted to a class of functions that can be expressed as a product of an operator monotone function and a Lipschitz continuous function. In this note, we remove these restrictions on the index function and establish optimal convergence rates for least-square regression over a Hilbert space with general regularization under a general source condition, thereby significantly broadening the scope of existing theoretical results.

\end{abstract}

\section{Introduction}
This article explores the least-square regression problem in a real 
separable Hilbert space, a fundamental setting in functional analysis and machine learning. In any regression task, understanding the smoothness of the target function plays a key role in determining the effectiveness and convergence of the chosen algorithm. In a Hilbert space, the smoothness of an element can be determined by the rate at which its Fourier coefficients decay with respect to a given orthonormal basis. This notion is typically used in the analysis of convergence  of methods for solving a compact linear operator equation
$Tf=g$, where the smoothness of the unknown solution is measured relative to the orthonormal eigenfunctions of the operator
$T$. In fact, the convergence rates of these methods inherently depend on the smoothness of the unknown solution. This naturally raises the question: how generally can we describe the smoothness of an element in a Hilbert space?  A beautiful result of \cite{howgeneral2008} answers that 
for any element $f$ in a Hilbert space, there exists an index function $\phi$—a continuous and strictly increasing function on $[0, s]$ with $\phi(0) = 0$—such that $f$ is in the range of $\phi(T)$.
Therefore, the most general approach to impose a smoothness condition on the target function is to require:  
$$f \in \Omega_{\phi}(R) := \{f: f = \phi(T)v, \|v\| \leq R\},$$
for some index function $\phi$ and positive constant $R$. 
This is referred as a general source condition on the target function.

A commonly studied source condition in Learning Theory literature is the Hölder source condition, characterized by $\phi(t) = t^r$ with $r > 0$ \cite{multipass, lin2020stochastic, devito, zhou2017distributed}. This has been extended to a wide class of source conditions, expressed via index functions with one of the three restrictive conditions, $(i)$ it is operator monotone, $(ii)$ it is Lipschitz continuous or $(iii)$ it can be written in the product of an operator monotone and a Lipschitz continuous function, 
which has been explored in works such as \cite{lin2020optimalspectral,sergeionregularization,rastogi2017optimal, shuai_2020_balancing}. 
While these restrictions simplify certain theoretical analyses, they also introduce significant challenges. First, verifying whether a specific function is operator monotone can be highly non-trivial, requiring deep knowledge of the function's properties and its interplay with operator theory. 
Second, while the class of index functions formed as the product of an operator monotone function and a Lipschitz function describes a rich set of examples, it does not encompass all index functions (see \cite{mathe2002_moduli}). For instance, non-smooth index functions can easily be constructed that are not expressible as such a product.
Consequently, this restrictive source condition might fail to encompass the true solution, potentially limiting the scope and applicability of the results.


In summary, the convergence analysis in Learning Theory is addressed under the source condition given by the product form of the index function as a Lipschitz and operator monotone function.
The fundamental question is whether the convergence rates can be derived for any index function so that the smoothness of any function can be addressed \cite{howgeneral2008}. In this paper, we affirmatively answer this fundamental question with  the optimal convergence rate for algorithms learning over a Hilbert space. As noted in \cite{lin2020optimalspectral}, our results also provide optimal convergence rates for classical learning theory framework with the general source condition. Therefore our work expands the applicability of regularized learning algorithms and provides a more comprehensive theoretical basis for studying convergence analysis in Learning Theory.

\subsection{Related work} Regularization methods in learning theory have been extensively studied. The classical works have been mainly focused on Tikhonov regularization \cite{poggio2000SVM, cuckerzhou2007learningtheory, smale2002foundation, smalezhouxuan2007learningtheory, sergei2013}, often under the Hölder smoothness assumption for the target function. In \cite{devito}, the authors established optimal convergence rates for Tikhonov regularization within the Hölder source condition framework. While learning over a Hilbert space, authors in \cite{multipass, lin2020stochastic, zhou2017distributed} explored general regularization framework to encompass the whole range of Hölder source condition. In \cite{sergeionregularization}, authors utilized this general regularization method to accommodate a wide range of source conditions represented by index functions that can be expressed as the product of an operator monotone function and a Lipschitz continuous function. Subsequently, \cite{rastogi2017optimal, shuai_2020_balancing} improved the results of \cite{sergeionregularization} and established optimal convergence rates within the same framework. More recently, the same source condition as \cite{sergeionregularization} has been explored in several studies \cite{lin2020optimalspectral, polynomial2023regularization}.



\subsection{Organization} The structure of the paper is as follows: Section~\ref{sec:preliminaries} introduces the essential background on the general regularization method within the framework of reproducing kernel Hilbert space (RKHS). Section~\ref{sec:main_results} outlines the key assumptions and presents the main findings of the study. Additionally, this section includes supplementary results required to support the proofs of the main results.

\section{Preliminaries}\label{sec:preliminaries}
Let us assume that input space $H$ is a real separable Hilbert space with inner product $\langle \cdot, \cdot \rangle_{H}$ and the output space $Y$ is a closed subset of $\mathbb{R}$. Let $\rho(w,y)= \rho_{H}(w)\rho(y|w)$ be a joint probability measure on $Z = H \times Y$, $\rho_{H}(\cdot)$ the induced marginal measure on $H$ having compact support, and $\rho(\cdot|w)$ the conditional probability measure on $Y$.\\
\noindent
The goal of least-square regression is to find an estimator that minimizes the expected risk
\begin{equation}\label{expected_risk}
   \inf_{w \in H}\mathcal{E}(w), \quad \mathcal{E}(w) = \int_{H \times Y} (\langle w, w' \rangle_{H}-y)^2 d\rho(w',y),
\end{equation}
where $\rho$ is known only through a sample $\textbf{z} := \{(w_{i}, y_{i})~:~ i=1,2,\ldots,n\}$ independently and identically distributed (i.i.d.) according to $\rho$. As the expected risk can not be computed exactly and that it can be only approximated through the empirical risk $\mathcal{E}_{\textbf{z}}(w)$, defined as:
\begin{equation}\label{vector_empirical_minimization_equation}
    \inf_{w \in H}\mathcal{E}_{\textbf{z}}(w), \quad\mathcal{E}_{\textbf{z}}(w) = \frac{1}{n}\sum_{i=1}^{n}(\langle w, w_{i}\rangle_{H}-y_{i})^2.
\end{equation}
Note that if we consider $H = L^2([0,1])$, then the problem $(\ref{expected_risk})$ provides a solution to the functional linear regression (FLR) model \cite{ramsay1991some}. Therefore, the problem under consideration is closely related to the FLR model.

\subsection{Reformulating the minimization problem in terms of operator equation} Let $L^2(H, \rho_H)$ denote the space of real valued square-integrable measurable functions on $H$, equipped with the norm $ \|\cdot\|_{\rho_H}$ and inner product $ \langle \cdot, \cdot \rangle_{\rho_H}$, defined with respect to the measure $ \rho_H$ and we assume that $\langle w, w'\rangle_{H} \leq \kappa^2,~ \rho_{H}-$ almost surely. We define operators $\mathcal{T}:= S_\rho^*S_{\rho} : H \to H$ and $\mathcal{L}:= S_\rho S_{\rho}^*: L^2(H, \rho_{H}) \to L^2(H, \rho_{H})$, where $S_{\rho}: H \to L^2(H, \rho_{H})$ is defined as $S_{\rho}(w) = \langle w, \cdot \rangle_{H}$ and $S_{\rho}^* : L^2(H, \rho_{H}) \to H$ is the adjoint of $S_{\rho}$ given as
$$S_{\rho}^* g = \int_{H} w g(w) d\rho_{H}(w).$$ 
Since the probability measure 
$\rho$ is not accessible, we consider a discretized versions of $S_{\rho}$, denoted as $\mathcal{S}_{n} : H \to \mathbb{R}^{n}$, which is defined as $(\mathcal{S}_{n}w)_{i} = \langle w, w_{i} \rangle_{H},~ i =1,\ldots,n$. We also define the operator $\mathcal{T}_{\mathbf{x}} = \mathcal{S}_{n}^*\mathcal{S}_{n} : H \to H$ given as:
$$\mathcal{T}_{\mathbf{x}}w = \frac{1}{n}\sum_{i=1}^{n} \langle w, w_{i} \rangle_{H} w_{i},$$
where $\mathcal{S}_{n}^*$ is the adjoint of $\mathcal{S}_{n}$. Consider $F(w) = \frac{1}{n}\sum_{i=1}^{n}(\langle w, w_{i}\rangle_{H}-y_{i})^2$, then the Fréchet derivative of $F$ is given as $DF(w)= \frac{1}{n}\sum_{i=1}^{n}w_{i}\langle w, w_{i}\rangle_{H} - \frac{1}{n}\sum_{i=1}^{n}y_{i}w_{i}$. As the minimizer will satisfy $DF(w)=0$, we can see that the minimizer of empirical risk in $(\ref{vector_empirical_minimization_equation})$ can be given by solving the linear equation $\mathcal{T}_{\mathbf{x}}w = \mathcal{S}_{n}^*y$. Solving this linear equation presents two main challenges. First, the operator \( \mathcal{T}_{\mathbf{x}} \) is not generally invertible, making the problem ill-posed. Second, the solution may overfit the data, which is not desired as the observed data need not be exact. To address these issues, regularization techniques are widely studied. So the estimator $w_{\textbf{z}}^{\lambda}$ of $w_{H}$ which is a minimizer of $(\ref{expected_risk})$ is given as $w_{\textbf{z}}^{\lambda} = g_{\lambda}(\mathcal{T}_{\mathbf{x}})\mathcal{S}_{n}^*y$, where $g_{\lambda}$ represents a regularization family, defined as follows.

\subsection{Regularization family}
A family of functions, $g_{\lambda}:[0,s] \to \mathbb{R},\, 0<\lambda \leq s$, is called the regularization family satisfying the following conditions:
\begin{itemize}
    \item There exists a constant $A>0$ such that
    \begin{equation}
    \label{reg1}
        \sup_{0<\sigma \leq s} |\sigma g_{\lambda}(\sigma )| \leq A.
    \end{equation}
    \item There exists a constant $B>0$ such that
    \begin{equation}
    \label{reg2}
        \sup_{0<\sigma \leq s} |g_{\lambda}(\sigma )| \leq \frac{B}{\lambda}.
    \end{equation}
    \item There exists a constant $D >0 $ such that
    \begin{equation}
        \label{reg3}
        \sup_{0<\sigma \leq s} |1- g_{\lambda}(\sigma )\sigma| = \sup_{0<\sigma \leq s} |r_{\lambda}(\sigma)| \leq D,
    \end{equation}
    where $r_{\lambda}(\sigma) = 1- \sigma g_{\lambda}(\sigma)$.
    \item The maximal $\nu$ such that 
    \begin{equation}
    \label{qualification}
    \sup_{0<\sigma \leq s} |1- g_{\lambda}(\sigma )\sigma|\sigma^{\nu} = \sup_{0<\sigma \leq s}|r_{\lambda}(\sigma)| \sigma^{\nu} \leq \gamma_{\nu} \lambda^{\nu},
    \end{equation}
    is called the qualification of the regularization family $g_{\lambda}$, where the constant $\gamma_{\nu}$ does not depend on $\lambda$.
\end{itemize}
\begin{example}
    Choosing $g_{\lambda}(\sigma) = \frac{1}{\sigma + \lambda}$ leads to Tikhonov regularization, which satisfies the aforementioned conditions with constants $A = B = D = 1$  and has a qualification of $1$. Another choice  of $g_{\lambda}(\sigma) = \sigma^{-1} \mathds{1}_{\sigma \geq \lambda}$, corresponding to spectral cut-off, also satisfies these conditions with $A = B = D = 1$ and possesses an arbitrary qualification.
\end{example}
Convergence analysis leverages the qualification of the regularization family by utilizing the relationship between the regularization family and the index function, as described in the following definition from \cite{mathe_2003_geometry}.
\begin{definition}
    We say that the qualification $\nu_{0}$ covers the index function $\phi$ if there is $c>0$ such that
\begin{equation}\label{qualification_covering_index}
    c \frac{t^{\nu_{0}}}{\phi(t)} \leq \inf_{t \leq \sigma \leq s} \frac{\sigma^{\nu_{0}}}{\phi(\sigma)},
\end{equation}
where $0<t \leq s$.
\end{definition}
\noindent
We refer the reader to \cite{sergei2013,englmartinbook,AI2022sergei} for more details of the regularization family and its interplay with general source condition.

%

\subsection{Relation with classical non-parametric regression} Consider the hypothesis space $\mathcal{H}_{K} = \{f : H \to \mathbb{R} ~|~ \exists ~ w\in H \text{ with } f(w') = \langle w, w' \rangle_{H},~ \rho_{H}-\text{almost surely}\}$ which is a reproducing kernel Hilbert space (RKHS) with the linear reproducing kernel $K : H \times H \to \mathbb{R}$ given as $K(w,w') = \langle w, w'\rangle_{H}$. This space is equipped with the inner product $\langle f, g \rangle_{\mathcal{H}_{K}} = \langle \langle w', \cdot \rangle_{H}, \langle w'', \cdot \rangle_{H} \rangle_{\mathcal{H}_{K}} = \langle w', w''\rangle_{H},~ \forall~f(\cdot) = \langle w', \cdot \rangle_{H},g = \langle w'', \cdot \rangle_{H} \in \mathcal{H}_{K}$. With this setup, the problem $(\ref{expected_risk})$ can be rewritten as:
\begin{equation}\label{expected_normal}
    \inf_{f \in \mathcal{H}_{K}}\tilde{\mathcal{E}}(f),\quad \tilde{\mathcal{E}}(f) = \int_{H \times \mathbb{R}} (f(w')-y)^2 d\rho(w',y).
\end{equation}
For further details on RKHS, we refer the reader to \cite{aronszajn1950rkhs,SVM2008steinwart,AI2022sergei}. Furthermore, assuming $\displaystyle\int_{H \times Y} y^2 \, d\rho(w, y) < \infty $, it follows that the regression function $f_{\rho} =\displaystyle\int_Y y \, d\rho(y|w)$ belongs to $L^2(H, \rho_H)$. Then it is straightforward to verify that $f_{\rho}$ is the function that minimizes $\tilde{\mathcal{E}}(f)$ over $L^2(H, \rho_H)$. Similar to the previous case, we replace the expected risk in problem $(\ref{expected_normal})$ with the empirical risk defined as follows: 
\begin{equation}\label{empirical_normal}
    \inf_{f \in \mathcal{H}_{K}}\tilde{\mathcal{E}_{\mathbf{z}}}(f), \quad \tilde{\mathcal{E}_{\mathbf{z}}}(f) = \frac{1}{n}\sum_{i=1}^{n}(f(w_{i})-y_{i})^2. 
\end{equation}

\noindent
We define operators $L_{K}:= I_{K}I_{K}^* : L^2(H, \rho_{H}) \to L^2(H, \rho_{H})$ and $T:= I_{K}^*I_{K}: \mathcal{H}_{K} \to \mathcal{H}_{K}$, where $I_{K}: \mathcal{H}_{K} \to L^2(H, \rho_{H}) $ is the inclusion operator defined as $(I_{K}f)(w') = \langle \langle w',\cdot \rangle_{H} ,f \rangle_{\mathcal{H}_{K}}$, 
 and $I_{K}^*:L^2(H, \rho_{H}) \to \mathcal{H}_{K}$ is the adjoint of the inclusion operator given as
$$I_{K}^*g= \int_{H} \langle w',\cdot \rangle_{H} g(w')\, d\rho_{H}(w').$$ 

Furthermore, we define the data-dependent operator $S_{\mathbf{x}} : \mathcal{H}_{K} \to \mathbb{R}^n$ as $(S_{\mathbf{x}}f)_{i} = f(w_{i}) = \langle f, \langle w_{i}, \cdot \rangle_{H} \rangle_{\mathcal{H}_{K}}, \forall i=1,2,\ldots,n$. The associated operator $T_{\mathbf{x}}$, given by $T_{\mathbf{x}} = S_{\mathbf{x}}^*S_{\mathbf{x}} : \mathcal{H}_{K} \to \mathcal{H}_{K}$ is explicitly defined as:
$$T_{\mathbf{x}}f = \frac{1}{n}\sum_{i=1}^{n} \langle w_{i}, \cdot \rangle_{H} \langle f, \langle w_{i},\cdot \rangle_{H} \rangle_{\mathcal{H}_{K}},$$
where $S_{\mathbf{x}}^*$ is the adjoint of $S_{\mathbf{x}}$.\\

\noindent
A solution $f_{\mathcal{H}}^{\dagger}$ to the problem $(\ref{expected_normal})$ is the projection of the regression function onto the closure of $\mathcal{H}_{K}$. Throughout the analysis we assume that the solution $f_{\mathcal{H}}^{\dagger} \in \mathcal{H}_{K}$. The objective is to construct an estimator $f_{\mathbf{z}}^{\lambda}$ of $f_{\mathcal{H}}^{\dagger}$ by solving $(\ref{empirical_normal})$.  Similar to problem $(\ref{vector_empirical_minimization_equation})$, a solution for problem $(\ref{empirical_normal})$ can be given by solving the linear equation $T_{\mathbf{x}}f = S_{\mathbf{x}}^*y$. Hence the estimator for the minimizer of $(\ref{expected_normal})$ is given as $f_{\mathbf{z}}^{\lambda}= g_{\lambda}(T_{\mathbf{x}})S_{\mathbf{x}}^*y$.\\

\noindent
Observe that the structure of the linear kernel establishes a correspondence between the input space $H$ and the RKHS $\mathcal{H}_{K}$ associated with the linear kernel. For each $w \in H$, we can define $K_w: H \to \mathbb{R}, K_w(w') = K(w, w')$, then $K_w \in \mathcal{H}_{K}$. Similarly, for any $f \in \mathcal{H}_{K}$, there exists $w \in H$ such that $f = K_w$. From this, it follows directly that the solution to the expected risk minimization problem $(\ref{expected_risk})$, denoted as $w_{H}$, satisfies the relationship $S_{\rho} w_{H} = f_{\mathcal{H}}^{\dagger}$. Similarly, when considering the empirical estimator of $w_{H}$, denoted by $ w_{\mathbf{z}}^{\lambda} $, satisfies the analogous equation $ S_{\rho} w_{\mathbf{z}}^{\lambda} = f_{\mathbf{z}}^{\lambda} $. Observe that $\|S_{\rho}(w)\|_{\mathcal{H}_{K}} = \|w\|_{H},~\forall~ w \in H$. Thus, the smoothness of the solution $w_{H}$ can be analyzed by studying the smoothness properties of $f_{\mathcal{H}}^{\dagger}$.\\

\noindent
\subsection{Smoothness condition}
Since the operators $T$ and $L_{K}$ are positive and compact, we have the following spectral representations:
$$T = \sum_{i=1}^{\infty} \mu_{i} \langle \cdot, e_{i} \rangle_{\mathcal{H}} e_{i}, \quad L_{K} = \sum_{i=1}^{\infty} \mu_{i} \langle \cdot, \psi_{i} \rangle_{\rho_{H}} \psi_{i},$$
where $\{(\mu_{i},e_{i})\}_{i \in \mathbb{N}}$ and $\{(\mu_{i},\psi_{i})\}_{i \in \mathbb{N}}$ are eigenvalue-eigenfunction pairs of operators $T$ and $L_{K}$ respectively with $\kappa^2 \geq \mu_{1}\geq \mu_{2} \geq \mu_{3}\geq \ldots > 0$. Using the spectral representation of operators, it is evident that $f_{\mathcal{H}}^{\dagger} = \sum_{i=1}^{\infty} \frac{1}{\sqrt{\mu_{i}}}\langle f_{\rho}, \psi_{i} \rangle_{\rho_{H}} e_{i} $. Then $f_{\mathcal{H}}^{\dagger} \in \mathcal{R}(S_{\rho}) = \mathcal{H}_{K}$ if and only if
\begin{equation}\label{picard_criterion}
    \sum_{i=1}^{\infty} \frac{\langle f_{\rho}, \psi_{i}\rangle_{\rho_{H}}^2}{\mu_{i}} < \infty.
\end{equation}
In other words, existence of a solution to $\inf_{f \in \mathcal{H}_{K}}\tilde{\mathcal{E}}(f)$ in $\mathcal{R}(S_{\rho})$ relies on the decay of the Fourier coefficients $\langle f_{\rho}, \psi_{i} \rangle_{\rho_{H}}$. Condition $(\ref{picard_criterion})$, referred as zero smoothness condition, is equivalent to stating that $f_{\mathcal{H}}^{\dagger} \in \mathcal{R}(\phi(T))$, where $\phi(t) = \sqrt{t}$. If we instead consider $\phi(t) = t^{\alpha}$ (Hölder source condition) with $\alpha \geq \frac{1}{2}$, this imposes a slightly stronger smoothness condition on $f_{\mathcal{H}}^{\dagger}$ requiring that
\begin{equation*}
    \sum_{i=1}^{\infty} \frac{\langle f_{\rho}, \psi_{i}\rangle_{\rho_{H}}^2}{\mu_{i}^{2 \alpha}} < \infty.
\end{equation*}
For instance if $\alpha = 1$, then the condition demands that the regression function $f_{\rho}$ should reside in $\mathcal{H}_{K}$.
In \cite{howgeneral2008}, the authors demonstrated that for every $f \in \mathcal{H}_{K}$ there exists an index function $\phi$ such that $f \in \mathcal{R}(\phi(T))$. Therefore, the natural way to describe the smoothness of $f_{\mathcal{H}}^{\dagger}$ is in terms of the decay of its Fourier coefficients, given by
\begin{equation}\label{sourceinfourier}
    \sum_{i=1}^{\infty} \frac{\langle f_{\rho}, \psi_{i} \rangle_{\rho_{H}}^2}{\phi^2(\mu_{i})} < \infty,
\end{equation}
where $\phi$ is an index function on the interval $[0, s]~(s \geq \kappa^2)$. 
\\

\noindent
So in terms of source condition, we assume that for some $R>0$
\begin{equation}\label{sourceset}
f_\mathcal{H}^{\dagger} \in \Omega_{\phi,R} := \{f \in \mathcal{H}_{K} : f = \phi(T)v, \|v\|_{\mathcal{H}} \leq R\}.
\end{equation}

\section{Main results}\label{sec:main_results}
In this section, we establish the convergence rates for the general regularization method under a general source condition, without imposing any additional constraints on the index function. We begin by outlining the assumptions necessary for our analysis.

\begin{assumption}\label{as:1}
We assume that  $\displaystyle\int_{H \times Y} y^2 \, d\rho(x, y) < \infty$. We also assume that $Pf_{\rho} \in \mathcal{R}(I_{K})$, where $P : L^2(X, \rho_{H}) \to L^2(X, \rho_{H})$ denotes the projection onto $\overline{\mathcal{R}(I_{K})}$.
\end{assumption}
As mentioned in the previous section, the assumption $\int_{H \times Y} y^2 \, d\rho(w, y) < \infty$ ensures that $f_{\rho} \in L^2(H, \rho_{H})$, while $Pf_{\rho} \in \mathcal{R}(I_{K})$ guarantees the existence of $f_{\mathcal{H}}^{\dagger}$, which is the Moore-Penrose inverse of the operator equation $I_{K}f = f_{\rho}$, and $ S_{\rho}w_{H} = f_{\mathcal{H}}^{\dagger}$. In \cite{lin2020optimalspectral}, the authors assume the existence of $w_H \in H$ that minimizes the expected risk in $(\ref{expected_risk})$. The presence of such a $w_H$ guarantees that the minimizer of the expected risk in $(\ref{expected_normal})$, given by 
$f_{\mathcal{H}}^{\dagger}$, belongs in $\mathcal{H}_{K}$. Consequently, assuming the existence of $w_H$ ensures that Assumption \ref{as:1} holds.

\begin{assumption}
    \label{as:2}
    We assume that $S_{\rho}w_{H} = f_{\mathcal{H}}^{\dagger} \in \Omega_{\phi, R} = \{f\in \mathcal{H}_{K} : f = \phi(T)v, \|v\|_{\mathcal{H}} \leq R \}$,
    where $\phi :[0,s] \to \mathbb{R}~ (s \geq \kappa^2)$ is an index function, i.e., $\phi$ is non-decreasing, continuous and $\phi(0) = 0$.
\end{assumption}

\begin{assumption}
    \label{as:3}
    There are two positive constants $\Sigma$ and $M$ such that for almost all $w \in H$,
    \begin{equation*}
        \int_{Y}\left(e^{\frac{|y-\langle w_{H},w \rangle_{H}|}{M}}-\frac{|y-\langle w_{H}, w \rangle_{H}|}{M}-1\right)d\rho(y|w) \leq \frac{\Sigma^2}{2M^2}.
    \end{equation*} 
\end{assumption}

\begin{assumption}
    \label{as:4}
    For some $b>1$,
\begin{equation*}
    i^{-b}\lesssim \mu_{i} \lesssim i^{-b} \quad \forall  i \in \mathbb{N},
\end{equation*}
where $(\mu_{i},e_{i})_{i\in \mathbb{N}}$ is the eigenvalue-eigenfunction pair of operator $T$ and the notation $a_{n} \lesssim b_{n}$ implies that there exists a positive constant $\tilde{c}$ such that $a_{n} \leq \tilde{c}b_{n}$ for all $n \in \mathbb{N}$.
\end{assumption}
Asymptotic polynomial decay condition of the eigenvalues $\{\mu_{i}\}_{i \in \mathbb{N}}$ ensures the bound on the effective dimension of the operator $T$ defined as 
$\mathcal{N}(\lambda)= \text{trace}(T (T+\lambda I)^{-1})$. Indeed,
\begin{equation}
\label{effectivedimention}
    \mathcal{N}(\lambda)= \sum_{i \in \mathbb{N}}\frac{\mu_{i}}{\mu_{i}+\lambda} \lesssim \sum_{i\in \mathbb{N}}\frac{i^{-b}}{i^{-b}+\lambda} \lesssim \lambda^{-\frac{1}{b}},
\end{equation}
where the last step follows from \cite[Lemma A.11]{gupta2024optimal}. Throughout this section, $\|U\|_{\mathcal{L}(\mathcal{H}_{K})}$ denotes the operator norm of a bounded operator $U: \mathcal{H}_{K} \to \mathcal{H}_{K}$.

\begin{theorem}\label{mainresult}
Assume Assumptions \ref{as:1}-\ref{as:3} are satisfied, and $\nu$ represents the qualification of the regularization family $\{g_{\lambda}\}$. Let $0<\eta<1$, $\sqrt{\lambda} \geq 16 \kappa \sqrt{\frac{\mathcal{N}(\lambda)}{n}} \log(4/\eta)$ and $\nu' = \lfloor \nu \rfloor$.
\begin{enumerate}
    \item If the index function $\phi$ is covered by the qualification of the regularization family, then with probability at least $1-\eta$, we have
    \begin{equation*}
         \|w_{\mathbf{z}}^{\lambda}- w_{H}\|_{H} \leq \tilde{C_{1}} \left(\frac{1}{n \lambda} + \sqrt{\frac{\mathcal{N}(\lambda)}{n \lambda}}\right)\log(8/\eta) + \tilde{C_{2}}\phi(\lambda),
    \end{equation*}
    where $\tilde{C_{1}} = \max\{C_{1},C_{2},C_{3},C_{4}\}$, $\tilde{C_{2}}= 2^{2\nu-\nu'-1} \max\{1,1/c\} (D+ \gamma_{\nu}) R$,
    $C_{1}=2^{2\nu-\nu'+1}(D+\gamma_{\nu})\nu' \kappa \|f_{\mathcal{H}}^{\dagger}\|_{\mathcal{H}_{K}}$, $C_{2} = 2 \sqrt{2} (B+\sqrt{AB}) \kappa M \|f_{\mathcal{H}}^{\dagger}\|_{\mathcal{H}_{K}}$, $C_{3}= 2^{2 \nu-\nu'} (D+\gamma_{\nu})\nu' \sqrt{\kappa} \|f_{\mathcal{H}}^{\dagger}\|_{\mathcal{H}_{K}}$ and $C_{4} = 2 \sqrt{2} (B+\sqrt{AB}) \Sigma \|f_{\mathcal{H}}^{\dagger}\|_{\mathcal{H}_{K}}$.\\
    \item If the index function $\sqrt{t} \phi(t)$ is covered by the qualification of the regularization family, then with probability at least $1-\eta$, we have
    \begin{equation*}
    \|\mathcal{T}^{\frac{1}{2}}(w_{\mathbf{z}}^{\lambda}- w_{H})\|_{H} \leq \tilde{C_{3}}\left(\frac{1}{n \sqrt{\lambda}}+ \sqrt{\frac{\mathcal{N(\lambda)}}{n}}\right) \log(8/\eta) + \tilde{C_{2}}\phi(\lambda) \sqrt{\lambda},
\end{equation*}
where $\tilde{C_{3}} = \max\{C_{1},C_{5},C_{3},C_{6}\}$, $C_{5} = 2 \sqrt{2} \nu_{3} \kappa M \|f_{\mathcal{H}}^{\dagger}\|_{\mathcal{H}_{K}}$, $C_{6} = 2 \sqrt{2} \nu_{3} \Sigma \|f_{\mathcal{H}}^{\dagger}\|_{\mathcal{H}_{K}}$, and $\nu_{3}= (B+A+2 \sqrt{AB})$.
\end{enumerate}    
\end{theorem}
\subsection{Proof of part(1) of Theorem \ref{mainresult}}
Observe that $$\|w_{\mathbf{z}}^{\lambda}- w_{H}\|^2_{H} = \|S_{\rho}w_{\mathbf{z}}^{\lambda} - S_{\rho}w_{H}\|^2_{\mathcal{H}_{K}} = \|\langle w_{\mathbf{z}}^{\lambda}, \cdot\rangle - \langle w_{H}, \cdot\rangle\|^2_{\mathcal{H}_{K}} =  \|f_{\mathbf{z}}^{\lambda}- f_{\mathcal{H}}^{\dagger}\|^2_{\mathcal{H}_{K}}.$$ Now we consider the error term
    \begin{equation*}
    \begin{split}
        \|f_{\mathbf{z}}^{\lambda}- f_{\mathcal{H}}^{\dagger}\|_{\mathcal{H}_{K}} \leq &  \underbrace{\|f_{\mathcal{H}}^{\dagger} - g_{\lambda}(T_{\mathbf{x}})T_{\mathbf{x}}f_{\mathcal{H}}^{\dagger}\|_{\mathcal{H}_{K}}}_{\textit{Term-1}} + \|g_{\lambda}(T_{\mathbf{x}})T_{\mathbf{x}}f_{\mathcal{H}}^{\dagger}-g_{\lambda}(T_{\mathbf{x}})S_{\mathbf{x}}^*y \|_{\mathcal{H}_{K}}.\\
    \end{split}
    \end{equation*}
Estimation of \textit{Term-1}: Consider
\begin{equation*}
    \begin{split}
         & \|f_{\mathcal{H}}^{\dagger} - g_{\lambda}(T_{\mathbf{x}})T_{\mathbf{x}}f_{\mathcal{H}}^{\dagger}\|_{\mathcal{H}_{K}} =  \|(I-g_{\lambda}(T_{\mathbf{x}})T_{\mathbf{x}})f_{\mathcal{H}}^{\dagger}\|_{\mathcal{H}_{K}}\\
        =& \|r_{\lambda}(T_{\mathbf{x}})\phi(T)v\|_{\mathcal{H}_{K}}
        = \|r_{\lambda}(T_{\mathbf{x}})(T_{\mathbf{x}}+\lambda I)^{\nu}(T_{\mathbf{x}}+\lambda I)^{-\nu}\phi(T)v\|_{\mathcal{H}_{K}}\\
        \stackrel{(*)}{\leq} & 2^{\nu-1} (D+\gamma_{\nu}) \lambda^{\nu} \|(T_{\mathbf{x}}+\lambda I)^{-\nu}\phi(T)v\|_{\mathcal{H}_{K}}\\
        = & 2^{\nu-1} (D+\gamma_{\nu}) \lambda^{\nu} \|(T_{\mathbf{x}}+\lambda I)^{-(\nu-\nu')}(T+\lambda I)^{(\nu-\nu')}(T+\lambda I)^{-(\nu-\nu')}(T_{\mathbf{x}}+\lambda I)^{-\nu'}\phi(T)v\|_{\mathcal{H}_{K}}\\
        \stackrel{(\dagger)}{\leq} & 2^{2\nu-\nu'-1} (D+\gamma_{\nu}) \lambda^{\nu} \|(T+\lambda I)^{-(\nu-\nu')}(T_{\mathbf{x}}+\lambda I)^{-\nu'}\phi(T)v\|_{\mathcal{H}_{K}}\\
        = & 2^{2\nu-\nu'-1} (D+\gamma_{\nu}) \lambda^{\nu} \|(T+\lambda I)^{-(\nu-\nu')}(T_{\mathbf{x}}+\lambda I)^{-\nu'}\phi(T)v \\
        & \qquad \qquad \qquad \qquad \qquad - (T+\lambda I)^{-\nu}\phi(T)v + (T+\lambda I)^{-\nu}\phi(T)v\|_{\mathcal{H}_{K}}\\
        \leq & 2^{2\nu-\nu'-1} (D+\gamma_{\nu})\left[\lambda^{\nu'} \underbrace{\|((T_{\mathbf{x}}+\lambda I)^{-\nu'}-(T+\lambda I)^{-\nu'})\phi(T)v\|_{\mathcal{H}_{K}}}_{\textit{Term-1(a)}} + \lambda^{\nu}\underbrace{\|(T+\lambda I)^{-\nu}\phi(T)v\|_{\mathcal{H}_{K}}}_{\textit{Term-1(b)}}\right],
    \end{split}
\end{equation*}
where we used $(\ref{reg3})$ and $(\ref{qualification})$ in $(*)$ and Lemma \ref{powerscomp} in $(\dagger)$.\\

\noindent
Bound of \textit{Term-1(a)}: Using Lemma \ref{reducing}, we get
\begin{equation*}
    \begin{split}
         & \|((T_{\mathbf{x}}+\lambda I)^{-\nu'}-(T+\lambda I)^{-\nu'})\phi(T)v\|_{\mathcal{H}_{K}} \\
        \leq & \|(T_{\mathbf{x}}+\lambda I)^{-(\nu'-1)}((T_{\mathbf{x}}+\lambda I)^{-1}-(T+\lambda I)^{-1})\phi(T)v\|_{\mathcal{H}_{K}}\\
        & \qquad \qquad \qquad \qquad + \|\sum_{i=1}^{\nu'-1}(T_{\mathbf{x}}+\lambda I)^{-i}(T-T_{\mathbf{x}})(T+\lambda I)^{-(\nu'+1-i)}\phi(T)v\|_{\mathcal{H}_{K}}\\
       \stackrel{(*)}{\leq} &\frac{1}{\lambda^{\nu'}}\|(T-T_{\mathbf{x}})(T+\lambda I)^{-1}\|_{\mathcal{L}(\mathcal{H}_{K})}\|f_{\mathcal{H}}^{\dagger}\|_{\mathcal{H}_{K}} + \sum_{i=1}^{\nu'-1}\frac{1}{\lambda^{i+\nu'-i}}\|(T-T_{\mathbf{x}})(T+\lambda I)^{-1}\|_{\mathcal{L}(\mathcal{H}_{K})}\|f_{\mathcal{H}}^{\dagger}\|_{\mathcal{H}_{K}}\\
        = & \frac{\nu'}{\lambda^{\nu'}}\|(T-T_{\mathbf{x}})(T+\lambda I)^{-1}\|_{\mathcal{L}(\mathcal{H}_{K})}\|f_{\mathcal{H}}^{\dagger}\|_{\mathcal{H}_{K}},
    \end{split}
\end{equation*}
where $(*)$ follows from the fact that $\lambda (T+\lambda I)^{-1} = I - T (T+ \lambda I)^{-1}$. Applying Lemma \ref{emp}, with a probability of at least $1 - \frac{\eta}{2}$, we have
\begin{equation*}
    \|((T_{\mathbf{x}}+\lambda I)^{-\nu'}-(T+\lambda I)^{-\nu'})\phi(T)v\|_{\mathcal{H}_{K}} \leq  \frac{2 \nu'  \|f_{\mathcal{H}}^{\dagger}\|_{\mathcal{H}_{K}
    }\log(4/\eta)}{\lambda^{\nu'}}\left(\frac{2 \kappa}{n \lambda}+ \sqrt{\frac{\kappa \mathcal{N}(\lambda)}{n \lambda}}\right). 
\end{equation*}
Bound of \textit{Term-1(b)}: Using Lemma \ref{source_cover_qualification}, we have
\begin{equation*}
    \begin{split}
        \|(T+\lambda I)^{-\nu}\phi(T)v\|_{\mathcal{H}_{K}} \leq & R \|(T+\lambda I)^{-\nu}\phi(T)\|_{\mathcal{L}(\mathcal{H}_{K})}\\
        \leq & R \sup_{0 \leq \sigma \leq \kappa^2 }\left|\frac{\phi(\sigma)}{(\sigma+\lambda)^{\nu}}\right| \leq \max\{1,1/c\} R \frac{\phi(\lambda)}{\lambda^{\nu}}.
    \end{split}
\end{equation*}
\noindent
Now from the analysis in \cite{rastogi2017optimal}, with probability at least $1-\frac{\eta}{2}$, we can see that
\begin{equation*}
    \|g_{\lambda}(T_{\mathbf{x}})T_{\mathbf{x}}f_{\mathcal{H}}^{\dagger}-g_{\lambda}(T_{\mathbf{x}})S_{\mathbf{x}}^*y\|_{\mathcal{H}_{K}} \leq 2\sqrt{2}(B+\sqrt{AB}) \left(\frac{\kappa M}{n\lambda} + \sqrt{\frac{\Sigma^2 \mathcal{N}(\lambda)}{n \lambda}}\right)\log\left(\frac{8}{\eta}\right).
\end{equation*}
Combining all the bounds, with probability at least $1-\eta$, we get
\begin{equation*}
    \|w_{\mathbf{z}}^{\lambda}- w_{H}\|^2_{H} \leq \tilde{C_{1}} \left(\frac{1}{n \lambda} + \sqrt{\frac{\mathcal{N}(\lambda)}{n \lambda}}\right)\log(8/\eta) + \tilde{C_{2}}\phi(\lambda).
\end{equation*}

\subsection{Proof of part(2) of Theorem \ref{mainresult}}
Observe that $$\|\mathcal{T}^{\frac{1}{2}}(w_{\mathbf{z}}^{\lambda}-w_{H})\|_{H} = \|S_{\rho}(w_{\mathbf{z}}^{\lambda}-w_{H})\|_{\rho} = \|f_{\mathbf{z}}^{\lambda}- f_{\mathcal{H}}^{\dagger}\|_{\rho}.$$ Now we consider the error term
\begin{equation*}
    \begin{split}
        \|f_{\mathbf{z}}^{\lambda}- f_{\mathcal{H}}^{\dagger}\|_{\rho}= &\|\sqrt{T}(f_{\mathbf{z}}^{\lambda}- f_{\mathcal{H}}^{\dagger})\|_{\mathcal{H}_{K}}\\
        \leq &  \underbrace{\|\sqrt{T}(f_{\mathcal{H}}^{\dagger} - g_{\lambda}(T_{\mathbf{x}})T_{\mathbf{x}}f_{\mathcal{H}}^{\dagger})\|_{\mathcal{H}_{K}}}_{\textit{Term-3}} + \|\sqrt{T}(g_{\lambda}(T_{\mathbf{x}})T_{\mathbf{x}}f_{\mathcal{H}}^{\dagger}-g_{\lambda}(T_{\mathbf{x}})S_{\mathbf{x}}^*y) \|_{\mathcal{H}_{K}}.\\
    \end{split}
    \end{equation*}
Again it follows from the analysis of \cite{rastogi2017optimal}, with probability at least $1-\frac{\eta}{2}$, we have
\begin{equation*}
    \|\sqrt{T}(g_{\lambda}(T_{\mathbf{x}})T_{\mathbf{x}}f_{\mathcal{H}}^{\dagger}-g_{\lambda}(T_{\mathbf{x}})S_{\mathbf{x}}^*y)\|_{\mathcal{H}_{K}} \leq 2\sqrt{2}\nu_{3} \left(\frac{\kappa M}{n\sqrt{\lambda}} + \sqrt{\frac{\Sigma^2 \mathcal{N}(\lambda)}{n}}\right)\log\left(\frac{8}{\eta}\right),
\end{equation*}
where $\nu_{3} = B+A+ 2 \sqrt{AB}$.\\

\noindent
Estimation of \textit{Term-3}: Let $\nu'' = \lfloor \nu-\frac{1}{2} \rfloor$. Consider
\begin{equation*}
    \begin{split}
        &\|\sqrt{T}(f_{\mathcal{H}}^{\dagger} - g_{\lambda}(T_{\mathbf{x}})T_{\mathbf{x}}f_{\mathcal{H}}^{\dagger})\|_{\mathcal{H}_{K}} = \|\sqrt{T}r_{\lambda}(T_{\mathbf{x}})f_{\mathcal{H}}^{\dagger}\|_{\mathcal{H}_{K}}\\
        \leq & \|\sqrt{T}(T+\lambda I)^{-\frac{1}{2}}\|_{\mathcal{L}(\mathcal{H}_{K})}\|(T+\lambda I)^{\frac{1}{2}}(T_{\mathbf{x}}+\lambda I)^{-\frac{1}{2}}\|_{\mathcal{L}(\mathcal{H}_{K})} \|(T_{\mathbf{x}}+\lambda I)^{\frac{1}{2}} r_{\lambda}(T_{\mathbf{x}})f_{\mathcal{H}}^{\dagger}\|_{\mathcal{H}_{K}}\\
        \stackrel{(*)}{\leq} & \sqrt{2} \|(T_{\mathbf{x}}+\lambda I)^{\frac{1}{2}} r_{\lambda}(T_{\mathbf{x}})(T_{\mathbf{x}}+\lambda I)^{\nu-\frac{1}{2}}(T_{\mathbf{x}}+\lambda I)^{-(\nu-\frac{1}{2})}f_{\mathcal{H}}^{\dagger}\|_{\mathcal{H}_{K}}\\
        \stackrel{(\dagger)}{\leq} & \sqrt{2} 2^{\nu-1} (D + \gamma_{\nu}) \lambda^{\nu}\|(T_{\mathbf{x}}+\lambda I)^{-(\nu-\frac{1}{2})}\phi(T)\nu\|_{\mathcal{H}_{K}},
    \end{split}
\end{equation*}
where we used Lemma \ref{powerscomp} in $(*)$ and $(\ref{reg3})$, $(\ref{qualification})$ in $(\dagger)$.\\

\noindent
Following the similar steps of Estimation of \textit{Term-1}, with probability at least $1-\frac{\eta}{2}$, we have
\begin{equation*}
    \begin{split}
        \|\sqrt{T}(f_{\mathcal{H}}^{\dagger} - g_{\lambda}(T_{\mathbf{x}})T_{\mathbf{x}}f_{\mathcal{H}}^{\dagger})\|_{\mathcal{H}_{K}} \leq \tilde{C_{2}}\phi(\lambda) \sqrt{\lambda} + \left(\frac{C_{1}}{n \sqrt{\lambda}}+ C_{3}\sqrt{\frac{ \mathcal{N}(\lambda)}{n }}\right)\log(4/\eta).
    \end{split}
\end{equation*}
Combining all the bounds, with probability at least $1-\eta$, we get
\begin{equation*}
    \begin{split}
        \|\mathcal{T}^{\frac{1}{2}}(w_{\mathbf{z}}^{\lambda}-w_{H})\|_{H} \leq \tilde{C_{3}}\left(\frac{1}{n \sqrt{\lambda}}+ \sqrt{\frac{\mathcal{N(\lambda)}}{n}}\right) \log(8/\eta) + \tilde{C_{2}}\phi(\lambda) \sqrt{\lambda}.
    \end{split}
\end{equation*}

\begin{corollary}\label{optimal_rates}
    Suppose Assumptions \ref{as:1}-\ref{as:4} hold and $0< \eta <1$. Assume that $\lambda \in (0,1], \lambda = \Psi^{-1}(n^{-\frac{1}{2}})$, where $\Psi(t) = t^{\frac{1}{2}+\frac{1}{2b}}\phi(t)$.
    \begin{enumerate}
    \item If the index function $\phi$ is covered by the qualification of the regularization family, then with probability at least $1-\eta$, we have
    \begin{equation*}
         \|w_{\mathbf{z}}^{\lambda}- w_{H}\|_{H} \leq c_{1} \phi(\Psi^{-1}(n^{-\frac{1}{2}}))\log(8/\eta),
    \end{equation*}
    for some positive constant $c_{1}$,
    \item If the function $\sqrt{t} \phi(t)$ is covered by the qualification of the regularization family, then with probability at least $1-\eta$, we have
    \begin{equation*}
    \|\mathcal{T}^{\frac{1}{2}}(w_{\mathbf{z}}^{\lambda}- w_{H})\|_{H} \leq c_{2}(\Psi^{-1}(n^{-\frac{1}{2}}))^{\frac{1}{2}}\phi(\Psi^{-1}(n^{-\frac{1}{2}}))\log(8/\eta),
\end{equation*}
for some positive constant $c_{2}$.
    \end{enumerate}
\end{corollary}

\begin{remark}
From Corollary \ref{optimal_rates}, it is evident that the convergence rates derived in this analysis are optimal, as they align with the lower bounds established in \cite{rastogi2017optimal}. By eliminating additional conditions on the index function, this work encompasses all scenarios, whether or not the index function can be expressed as a product of an operator monotone and a Lipschitz function. Consequently, this analysis extends and generalizes all the existing results in the framework discussed in \cite{devito, sergeionregularization, rastogi2017optimal, shuai_2020_balancing,lin2020optimalspectral}.
\end{remark}

In our analysis, we primarily work with the linear reproducing kernel which provides a bridge between the FLR model and the classical learning theory framework. We solve problems $(\ref{expected_risk})$ and $(\ref{expected_normal})$ simultaneously by utilizing this connection given by the structure of the linear kernel. However, our approach is not restricted to the linear kernel alone. We can see from Corollary \ref{general_framework} that the same theoretical framework and techniques can be applied to any reproducing kernel $K$ that satisfies the necessary assumptions. Specifically, if $K$ is a Mercer kernel with eigenvalues $\mu_{i},~ i \in \mathbb{N}$ satisfying appropriate polynomial decay conditions, then our results remain valid. This generalization is crucial as it ensures that our theoretical findings are not merely confined to a special case but rather extend to a broad class of kernels used in learning theory framework.

\begin{corollary}\label{general_framework}
    Let $K:H \times H \to \mathbb{R}$ be a Mercer kernel such that $\sqrt{K(x,x)}\leq \kappa,~\rho_{H}-$ almost surely with RKHS $(\mathcal{H}, \langle \cdot, \cdot\rangle_{\mathcal{H}})$. Suppose Assumptions \ref{as:1}-\ref{as:4} hold and $0 < \eta <1$. Assume that $\lambda \in (0,1],~\lambda = \Psi^{-1}(n^{-\frac{1}{2}})$, where $\Psi(t) = t^{\frac{1}{2}+\frac{1}{2b}}\phi(t)$.
\begin{enumerate}
    \item If the index function $\phi$ is covered by the qualification of the regularization family, then with probability at least $1-\eta$, we have
    \begin{equation*}
         \|f_{\mathbf{z}}^{\lambda}- f_{\mathcal{H}}^{\dagger}\|_{\mathcal{H}} \leq c_{1} \phi(\Psi^{-1}(n^{-\frac{1}{2}}))\log(8/\eta),
    \end{equation*}
    for some positive constant $c_{1}$,
    \item If the function $\sqrt{t} \phi(t)$ is covered by the qualification of the regularization family, then with probability at least $1-\eta$, we have
    \begin{equation*}
    \|f_{\mathbf{z}}^{\lambda}- f_{\mathcal{H}}^{\dagger}\|_{\rho_{H}} \leq c_{2}(\Psi^{-1}(n^{-\frac{1}{2}}))^{\frac{1}{2}}\phi(\Psi^{-1}(n^{-\frac{1}{2}}))\log(8/\eta),
\end{equation*}
for some positive constant $c_{2}$.
    \end{enumerate}
\end{corollary}

\noindent
We now outline the lemmas utilized in the proof of the main theorem.

\begin{lemma}\cite{devito}\label{emp}
Let $n \in \mathbb{N}$ and $0 < \eta <1$, then with probability at least $1-\frac{\eta}{2}$, we have
    \begin{equation*}
        \|(T+\lambda I)^{-1}(T-T_{\mathbf{x}})\|_{\mathcal{L}(\mathcal{H}_{K})} \leq 2 \log(4/\eta)\left(\frac{2 \kappa}{n \lambda}+ \sqrt{\frac{\kappa \mathcal{N}(\lambda)}{n \lambda}}\right).
    \end{equation*}
\end{lemma}

\begin{lemma}\label{source_cover_qualification}
Let $s > \kappa^2$ and $\phi: [0,s] \to \mathbb{R}$ be an index function. Then 
\begin{equation*}
    \sup_{0 \le \sigma \leq s} \left|\frac{\varphi(\sigma)}{(\sigma+\lambda)^{\nu}}\right| \leq \max\{1,1/c\}\frac{\varphi(\lambda)}{\lambda^{\nu}}, ~~ \forall \lambda \in (0,s].
\end{equation*}
\end{lemma}
\begin{proof}
    First we consider the case when $0 \leq \sigma < \lambda$
\begin{equation*}
    \begin{split}
        \sup_{0 \leq \sigma < \lambda} \left|\frac{\varphi(\sigma)}{(\sigma+\lambda)^{\nu}}\right| \leq \sup_{0 \leq \sigma < \lambda} \left| \frac{\varphi(\sigma)}{\lambda^{\nu}} \right| \leq \frac{\varphi(\lambda)}{\lambda^{\nu}}.
    \end{split}    
\end{equation*}
Next, we take the case when $\lambda \leq \sigma \leq s$
\begin{equation*}
    \begin{split}
         \sup_{\lambda \leq \sigma \leq s} \left|\frac{\varphi(\sigma)}{(\sigma+\lambda)^{\nu}}\right| = & \sup_{\lambda \leq \sigma \leq s} \left|\frac{\varphi(\sigma)}{(\sigma+\lambda)^{\nu}} \frac{\sigma^{\nu}}{\sigma^{\nu}}\right|\\
         \leq & \sup_{\lambda \leq \sigma \leq s}\left|\frac{\varphi(\sigma)}{\sigma^{\nu}} \right| \leq \frac{1}{c} \frac{\varphi(\lambda)}{\lambda^{\nu}}.
    \end{split}
\end{equation*}
\end{proof}

\begin{lemma}\cite[Lemma 5.1]{cordes1987}
\label{cordes}
Suppose $T_1$ and $T_2$ are two positive bounded linear operators on a separable Hilbert space. Then
$$\|T_1^pT_2^p\| \leq \|T_1T_2\|^p, \text{ when } 0\leq p \leq 1. $$
\end{lemma}

\begin{lemma}\cite[Lemma A.3]{gupta2024optimal}
  \label{reducing}
  Let $\Lambda$ and $\hat{\Lambda}_{n}$ be positive operators from $H_{1}$ to $H_{2}$, then for any $n \geq 1$, we have
    \begin{equation*}
    \begin{split}
        (\hat{\Lambda}_{n}+\lambda I )^{-n} - (\Lambda+\lambda I )^{-n} =  &(\hat{\Lambda}_{n}+\lambda I )^{-(n-1)}[(\hat{\Lambda}_{n}+\lambda I )^{-1} - (\Lambda+\lambda I )^{-1}]\\
        & + \sum_{i=1}^{n-1}(\hat{\Lambda}_{n}+\lambda I )^{-i}(\Lambda-\hat{\Lambda}_{n})(\Lambda+\lambda I )^{-(n+1-i)}.
    \end{split}
\end{equation*}
\end{lemma}

\begin{lemma}\label{powerscomp}
Let $0 < l,\eta <1$, then with probability at least $1-\frac{\eta}{2}$, we have
    \begin{equation*}
        \|(T+\lambda I)^{l}(T_{\mathbf{x}}+\lambda I)^{-l}\|_{\mathcal{L}(\mathcal{H}_{K})} \leq 2^l ,
    \end{equation*}
    provided
\begin{equation}\label{lambdabound}
    \sqrt{\lambda} \geq 16 \kappa \sqrt{\frac{\mathcal{N}(\lambda)}{n}} \log(4/\eta).
\end{equation}
\end{lemma}
\begin{proof}
From Lemma \ref{emp} and condition $(\ref{lambdabound})$, we know that
\begin{equation*}
    \|(T_{\mathbf{x}} - T)(T+ \lambda I)^{-1}\|_{\mathcal{L}(\mathcal{H}_{K})} \leq 8\kappa \sqrt{\frac{\mathcal{N(\lambda)}}{n \lambda}}\log(4/\eta) \leq 1/2.
\end{equation*}
Next, we have
    \begin{equation}\label{power1}
    \begin{split}
        \|(T+\lambda I)(T_{\mathbf{x}}+\lambda I)^{-1}\|_{\mathcal{L}(\mathcal{H}_{K})} = & \|[I + (T_{\mathbf{x}} - T)(T+ \lambda I)^{-1}]^{-1}\|_{\mathcal{L}(\mathcal{H}_{K})}\\
        \leq & \frac{1}{1-\|(T_{\mathbf{x}} - T)(T+ \lambda I)^{-1}\|_{\mathcal{L}(\mathcal{H}_{K})}} \leq 2,
    \end{split}
    \end{equation}
and the result follows from Lemma \ref{cordes}.
\end{proof}

\section*{Acknowledgment}
SS acknowledges the Anusandhan National Research Foundation, Government of India, for the financial
support through MATRICS grant - MTR/2022/000383.
\bibliographystyle{acm}
\bibliography{123}
\end{document}